\documentclass{article}
\let\tldocenglish=1  
\usepackage{tex-live}
\usepackage[latin1]{inputenc} 
\usepackage[T1]{fontenc}

\usepackage[inactive]{srcltx}
\usepackage[centertags]{amsmath}
\usepackage{amsfonts}
\usepackage{amssymb}
\usepackage{amsthm}
\usepackage{newlfont}

\theoremstyle{definition}
\newtheorem{definition}{Definition}[section]
\theoremstyle{remark} \theoremstyle{theorem}
\newtheorem{remark}{Remark}[section]
\newtheorem{theorem}{Theorem}[section]
\newtheorem{lemma}{Lemma}[section]
\newtheorem{corollary}{Corollary}[section]

\title{Existence, proper Pareto reducibility, and connectedness in multi-objective optimization}

\author{Latif Pourkarimi\\[3mm]
Razi University, Kermanshah, Iran;\\[1mm]
\url{lp_karimi@yahoo.com}\\[3mm]
$\&$\\[3mm]
Majid Soleimani-damaneh\\[3mm]
University of Tehran, Tehran, Iran;\\[1mm]
 \url{soleimani@khayam.ut.ac.ir}}

\date{April 2018}

\begin{document}
\maketitle

\begin{abstract}
This paper is divided to two parts. In the first part, we provide
elementary proofs for some important results in multi-objective
optimization. The given proofs are so simple and short in compared
to the existing ones. Also, a Pareto reducibility result is
extended from efficiency to proper efficiency. The second part is
devoted to the relationships between nonemptiness,
$R^p_{\geqq}$-(semi)compactness, external stability and
connectedness of the set of nondominated solutions in
multi-objective optimization. Furthermore, it is shown that some
assumption in an important result, concerning connectedness, is
redundant and should be removed.\vspace{3mm}
\end{abstract}

\begin{multicols}{2}
\tableofcontents
\end{multicols}

\section{Introduction}
Multi-objective optimization refers to maximiaing/minimizing more
than one objective functions over a feasible set. The image of the
feasible set under the objective functions is called the image
space, and it is usually denoted by $Y$. The set of minimals of
$Y$ is denoted by $Y_N$. Two basic and important questions in
multi-objective optimization are asking about the conditions under
which $Y_N\neq\emptyset$ and $Y_N$ is externally stable (i.e. each
dominated point of $Y$ is dominated by a member of $Y_N$)
\cite{ehr,saw}. Another important result in multi-objective
optimization is representing the set of (weak) efficient solutions
of a multi-objective problem with respect to that of its
subproblems \cite{ehr,mal,saw}. This subject is called Pareto
reducibility. The $\lq\lq$Pareto
reducibility" term was first used by Popovici \cite{pop}. In the first part of this paper, we provide
elementary proofs for some important results concerning existence,
external stability, and Pareto reducibility. The given proofs are
so simple and short in compared to the existing ones, and
specially these are suitable for teaching purposes. Also, a result
about Pareto reducibility is extended from efficiency to proper efficiency.

Connectedness and $R^p_{\geqq}$-(semi)compactness of $Y_N$ are
also two important notions in multi-objective optimization
\cite{ehr,saw}. The second part of the paper establishes the
equivalence of nonemptiness, external stability,
$R^p_{\geqq}$-compactness, and $R^p_{\geqq}$-semicompactness of
$Y_N$ under appropriate assumptions. Furthermore, it is shown that
one of the assumptions of a well-known result concerning
connectedness (established in \cite{nac}) is redundant and should
be removed.

The preliminaries are given in Section 2 and the main results are
presented in Sections 3 and 4.

\section{Preliminaries}
For two vectors $y^0,y^*\in R^p$, we use the following
componentwise orders:\\
$\bullet~y^0\leqq y^*$ iff $y_j^0\leq y^*_j$ for each $j$;\\
$\bullet~y^0\leq y^*$ iff $y^0\leqq y^*$ and $y^0\neq y^*$;\\
$\bullet~y^0<y^*$ iff $y_j^0<y^*_j$ for each $j$.

Three orders $\geqq,~\geq,$ and $>$ are defined analogously. Using
the componentwise order $\geqq$, the following cone is defined,
which is called the natural ordering cone:
$$R^p_{\geqq}=\{y\in R^p~:~y\geqq 0\}.$$

\begin{definition}
Let $Y\subseteq R^p$.\\
(i) $y^0\in Y$ is called a nondominated point of $Y$ if there does not exist $y\in Y$ such that $y\leq y^0$;\\
(ii) $y^0\in Y$ is called a weakly nondominated point of $Y$ if there does not exist $y\in Y$ such that $y<y^0$.
\end{definition}

The set of all nondominated points and the set of all weakly
nondominated points of $Y$ are denoted by $Y_N$ and $Y_{WN},$
respectively.

\begin{definition} \cite{geo}
Let $Y\subseteq R^p$. The vector $y^0\in Y$ is called a properly
nondominated point of $Y$ if $y^0\in Y_N$ and there exists scalar
$M>0$ such that for each $y\in Y$ and $i\in \{1,2,\ldots,p\}$
satisfying $y_i<y_i^0$ there exists $j\in \{1,2,\ldots,p\}$ such
that $y_j>y_j^0$ and
$$\frac{y_i^0-y_i}{y_j-y_j^0}\leq M.$$
\end{definition}

The set of all properly nondominated points of $Y$ is denoted by
$Y_{PN}$.


\begin{definition} \label{concepts}\cite{ehr}
$Y\subseteq R^p$ is called \\
(i) $R^p_{\geqq}$-convex if $Y+R^p_{\geqq}$ is convex;\\
(ii) $R^p_{\geqq}$-closed if $Y+R^p_{\geqq}$ is closed; \\
(iii) $R^p_{\geqq}$-compact if $(y-R^p_{\geqq})\cap Y$ is compact for every $y\in Y$;\\
(v) $R^p_{\geqq}$-semicompact if every open cover of $Y$ of the
form $\{(y^i-R^p_{\geqq})^c~:~y^i\in Y,~i\in \mathcal{I}\}$ has a
finite subcover.\end{definition}

\begin{definition} \cite{ehr} Let $Y\subseteq R^p$. The set $Y_N$ is called externally stable if $Y\subseteq Y_N+R^p_{\geqq}.$\end{definition}

The set $C\subseteq R^p$ is a cone if $\lambda C\subseteq C$ for
each $\lambda\geq 0$. If furthermore, $C+C\subseteq C$, then it is
said to be a convex cone. If $C\cap (-C)=\{0\}$, then it is called
pointed. The cone $C$ is called proper if it is nonempty, $C\neq
\{0\}$, and $C\neq R^p$. The nonegative and positive polar cones
corresponding to $C$ are defined as follows, respectively:
$$C^+=\{d\in R^p~:~d^tx\geq 0,~\forall x\in C\},~~~~~~~~$$
$$C^{++}=\{d\in R^p~:~d^tx>0,~\forall x\in C\setminus \{0\}\}.$$


Consider a multi-objective optimization problem (MOP) as follows:
\begin{equation}\label{mop}\begin{array}{l}
  \min f(x)=(f_{1}(x),f_{2}(x),...,f_{p}(x)),\\
  s.t.~~~x\in X, \end{array}
\end{equation}
where $X\subseteq R^n$ is a nonempty set and $f$ is a
vector-valued function composed of $p\geq 2$ real-valued
functions. The image of $X$ under $f$ is denoted by
$Y:=f(X)\subseteq R^p$ and is referred to image space.

\begin{definition}\label{d:1}
A feasible solution $\hat{x}\in X$ is called \\
(i) an efficient solution to MOP (\ref{mop}) if there is no $x\in
X$ such that $f(x)\leq f(\hat{x})$;\\
(ii) a weakly efficient solution to MOP (\ref{mop}) if there is no
$x\in X$ such that $f(x)<f(\hat{x})$.
\end{definition}

The set of all efficient solutions and the set of all weakly
efficient solutions of MOP (\ref{mop}) are denoted by $X_E(f)$ and
$X_{WE}(f)$, respectively.

In order to obtain efficient solutions with bounded trade-offs,
Geoffrion \cite{geo} suggested restricting attention to efficient
solutions that are proper in the sense of the following
definition.
\begin{definition}\label{d: 2}\cite{geo} A feasible solution $\hat{x}\in X$ is called a properly efficient
solution to MOP (\ref{mop}) if it is efficient and there is a real number $M>0$ such that for all
$i\in \{1,2,...,p\}$ and $x\in X$ satisfying $f_{i}(x)<f_{i}(\hat{x})$ there exists an index
$j\in \{1,2,...,p\}$ such that $f_{j}(x)>f_{j}(\hat{x})$ and
$$\frac{f_{i}(\hat{x})-f_{i}(x)}{f_{j}(x)-f_{j}(\hat{x})}\leq M.$$
\end{definition}

The set of all properly efficient solutions of MOP (\ref{mop}) is
denoted by $X_{PE}(f)$. 


Let $\rho\subseteq \{1,2,\ldots,p\}$ be nonempty. The set of all
efficient (resp. weakly efficient) solutions of
\begin{equation}\label{sub}\begin{array}{l}
  \min~f^\rho(x)=(f_{i}(x);~~i\in \rho)\\
  s.t.~~~x\in X, \end{array}
\end{equation}
is denoted by $X_{E}(f^{\rho})$ (resp. $X_{WE}(f^{\rho})$). The
set of all properly efficient solutions of Problem (\ref{sub}) is
denoted by $X_{PE}(f^{\rho})$.

\section{Some elementary proofs}
This section contains elementary proofs for three important
results in multi-objective optimization theory.

The following theorem has been proved by Borwein \cite{bor} for
general real linear vector spaces, and has been addressed in some
reference books, including \cite{ehr,saw}, for finite dimensional
multi-objective optimization. The proof addressed in
\cite{ehr,saw} is a technical proof utilizing the Zorn's lemma. In
this paper, we present a simple and short proof for this result.
The new proof is more appropriate for teaching purposes in
compared to the existing one.
\begin{theorem}\label{existence} \cite{ehr,saw} Let $Y\subseteq R^p$ and $y^0\in Y$. If $(y^0-R^p_{\geqq})\cap Y$ is compact, then $Y_N\neq \emptyset.$
\end{theorem}
\begin{proof} Consider the following auxiliary optimization problem:
\begin{equation}\label{aux}
\begin{array}{llr}
    \min & \displaystyle\sum_{j=1}^p y_j&\\
    s.t. & y\leqq y^0,&\\
    & y\in Y.&
  \end{array}
  \end{equation}
The set of feasible solutions of Problem (\ref{aux}) is $(y^0-R^p_{\geqq})\cap Y$ which is compact. The objective function of this problem
is continuous. Therefore, Problem (\ref{aux}) has an optimal solution, say $y^*$. We show that $y^*\in Y_N$. If $y^*\notin Y_N$, then
$$\exists \bar{y}\in Y;~~\bar{y}\leqq y^*\leqq y^0,~~\bar{y}\neq y^*.$$
Therefore, $\bar{y}$ is a feasible solution to (\ref{aux}) and
$$\displaystyle\sum_{j=1}^p \bar{y}_j<\displaystyle\sum_{j=1}^p y_j^*.$$
This contradicts the optimality of $y^*$ for (\ref{aux}) and completes the proof.  \end{proof}
\begin{remark}\label{remrem} Although in this paper we considered the natural cone $R^p_{\geqq}$ for ordering $R^p$, the above result is still valid if one considers any pointed convex
closed proper cone $C\subseteq R^p$ in lieu of $R^p_{\geqq}$ (as can be found in \cite{luc}). It can be proved utilizing our simple approach as well. To
this end, consider $d\in C^{++}$ and construct the following
auxiliary problem\footnote{It is not difficult to see that
$C^{++}\neq \emptyset$; see Remark 1.6 and Proposition 1.10 in
\cite{luc}.}:
\begin{equation}\label{aux-cone}\begin{array}{llrr}
    \min & d^ty&& \\
    s.t. & y^0-y\in C,&&\\
    & y\in Y.&&
  \end{array}
  \end{equation}
The set of feasible solutions of Problem (\ref{aux-cone}) is the compact set $(y^0-C)\cap Y$ and its objective function is continuous.
Therefore, Problem (\ref{aux-cone}) has an optimal solution, say $y^*$. If $y^*\notin Y_N$, then
$$\exists \bar{y}\in Y;~~y^*-\bar{y}\in C\backslash \{0\}$$
$$\Longrightarrow d^t(y^*-\bar{y})>0\Longrightarrow d^ty^*>d^t\bar{y}.$$
On the other hand, $$y^0-\bar{y}=y^0-y^*+y^*-\bar{y}\in C+C=C.$$
Hence, $\bar{y}$ is a feasible solution to Problem
(\ref{aux-cone}) and contradicts the optimality of $y^*$.\end{remark}

The theorem below has been addressed in some reference books,
including \cite{ehr,saw}, on finite dimensional multi-objective
optimization. In the following, we present a simple and short
proof for this result.
\begin{theorem}\label{stab}\cite{ehr,saw} Let $Y\subseteq R^p$ be nonempty and $R^p_{\geqq}-$compact. Then $Y_N$ is externally stable. \end{theorem}
\begin{proof} Consider arbitrary $y^0\in Y$ and the auxiliary
optimization problem (\ref{aux}) as defined in the proof of
Theorem \ref{existence}. Similar to the proof of Theorem
\ref{existence}, it can be seen that Problem (\ref{aux}) has an
optimal solution, say $y^*\in Y_N$. Hence $y^*\in
y^0-R^p_{\geqq}.$ This implies
$$y^0\in y^*+R^p_{\geqq}\subseteq Y_N+R^p_{\geqq}.$$
Thus, $$Y\subseteq Y_N+R^p_{\geqq},$$ because $y^0\in Y$ was
arbitrary. This completes the proof.  \end{proof}
\begin{remark} The above result can be proved in a similar
way if one considers any pointed convex closed proper cone
$C\subseteq R^p$ instead of $R^p_{\geqq}$. See Remark \ref{remrem}
for more detail.\end{remark}

The rest of this section is devoted to Pareto reducibility \cite{pop}.
The following theorem gives a simple proof for Proposition 2.35 in
\cite{ehr} (see also Malivert and Boissard \cite{mal} and Lowe et al \cite{low}). Malivert
and Boissard \cite{mal} proved this result for representing the
weak efficient set of (MOP) with respect to the efficient set of
its subproblems. The proof given in \cite{ehr,mal} is very
complicated in compared to that given in the present paper.
\begin{theorem}\label{union} Assume that $f_1,f_2,\ldots,f_k$ are convex
functions and the feasible set, $X$, is convex. Then
$$X_{WE}(f)=\bigcup_{\begin{array}{c}
  \rho\subseteq \{1,2,\ldots,p\}\vspace{-1mm}\\
  \rho\neq \emptyset
\end{array}}X_E(f^{\rho}).$$\end{theorem}

\begin{proof} $\supseteq:$ The proof of this part is trivial.\\
$\subseteq:$ The convexity assumptions imply that $Y=f(X)$ is
$R^p_{\geqq}$-convex. Assume that $\hat{x}\in X_{WE}(f)$. By the
weight-sum scalarization technique (Theorem 3.5 in \cite{ehr}),
there exists a nonzero vector
$\Lambda=(\lambda_1,\lambda_2,\ldots,\lambda_p)\geqslant 0$ such
that $\hat{x}$ is an optimal solution to
$$\begin{array}{llrr}
    \min & \Lambda^tf(x)&~~~~~~~~~~~~~~~&\\
    s.t. & x\in X.&&
  \end{array}$$
Setting $\rho=\{i\in \{1,2,\ldots,p\}~:~\lambda_i>0\}$, we have
$\hat{x}\subseteq X_{PE}(f^{\rho})\subseteq X_E(f^{\rho})$
according to Theorem 3.11 in \cite{ehr}. This completes the proof.\end{proof}

In fact, in the proof of the above theorem, we proved the
following theorem as well. The following theorem extends Theorem
2.36 in \cite{ehr}. The equality provided in Theorem \ref{,} is stronger than that given in
Theorem \ref{union} in the present paper and Theorem 2.36 in
\cite{ehr}. Notice that this equality provides a representation
for weak efficient solutions with respect to the properly
efficient solutions.

\begin{theorem}\label{,} Assume that $f_1,f_2,\ldots,f_k$ are convex
functions and the feasible set, $X$, is convex. Then
\begin{equation}\label{weak-proper}
X_{WE}(f)=\bigcup_{\begin{array}{c}
  \rho\subseteq \{1,2,\ldots,p\}\vspace{-3mm}\\
  \rho\neq \emptyset
\end{array}}X_{PE}(f^{\rho}).\end{equation}
\end{theorem}

The function
$f:X\longrightarrow R^p$ is called convexlike if for each $x,y\in
X$ and each $\lambda\in (0,1)$ there exists some $z \in X$ such
that $f(z)\leq \lambda f(x)+(1-\lambda)f(y).$ It is not difficult
to see that; if $f$ is convex-like, then $Y=f(X)$ is
$R^p_{\geqq}$-convex. The following result shows that Equation
(\ref{weak-proper}) holds if one replaces the assumption $\lq\lq
f$ is convex" with weaker assumption $\lq\lq f$ is convexlike".
The proof of this theorem is similar to that of Theorem
\ref{union} and is hence omitted.

\begin{theorem}\label{I2} If $f $ is a convex-like function on $X$, then
$$X_{WE}(f)=\bigcup_{\emptyset\neq \rho\subseteq\{1,\dots,p\}}X_{PE}(f^{\rho}).$$
\end{theorem}

\section{External stability and connectedness}

The following theorem shows that the result given in Theorem \ref{stab} is still valid under weaker assumption $\lq\lq R^p_{\geqq}$-semicompactness" in lieu
of $\lq\lq R^p_{\geqq}$ -compactness".
\begin{theorem}\label{semi} If $Y\subseteq R^p$ is nonempty and $R^p_{\geqq}-$semicompact, then $Y_N$ is externally stable. \end{theorem}
\begin{proof} Let $y^0\in Y$ be arbitrary. Setting $$Y^0=(y^0-R^p_{\geqq})\cap Y,$$ we show that $Y^0$ is $R^p_{\geqq}-$semicompact. To this end, assume that
$$Y^0\subseteq \displaystyle\cup_{i\in I} (y^i-R^p_{\geqq})^c.$$
Then
$$Y\subseteq \biggl(\displaystyle\cup_{i\in I} (y^i-R^p_{\geqq})^c\biggl)\cup (y^0-R^p_{\geqq})^c.$$
By assumption of the theorem, there exists $m\in N$ such that
$$Y\subseteq \biggl(\displaystyle\cup_{i=1}^m (y^i-R^p_{\geqq})^c\biggl)\cup (y^0-R^p_{\geqq})^c.$$
Hence, $$Y^0=(y^0-R^p_{\geqq})\cap Y\subseteq \displaystyle\cup_{i=1}^m (y^i-R^p_{\geqq})^c.$$
Therefore, $Y^0$ is $R^p_{\geqq}-$semicompact.

Since $Y^0$ is $R^p_{\geqq}-$semicompact, by Theorem 2.12 in \cite{ehr}, $Y^0_N\neq \emptyset.$ Thus, there exists $\bar{y}\in Y_N^0$.\\
Now, we show that $\bar{y}\in Y_N$; otherwise $$\exists y^*\in Y;~~y^*\leqq \bar{y} \leqq y^0,~~y^*\neq \bar{y}.$$
$$\Longrightarrow \exists y^*\in Y^0,~~y^*\leqq \bar{y},~~y^*\neq \bar{y}.$$
These contradict $\bar{y}\in Y_N^0$. Therefore, $\bar{y}\in Y_N\cap Y_N^0.$
Hence, $$y^0\in \bar{y}+R^p_{\geqq}\subseteq Y_N+R^p_{\geqq}.$$
Since $y^0\in Y$ is arbitrary, we have
$$Y\subseteq Y_N+R^p_{\geqq},$$
and the proof is completed.  \end{proof}

For a given set $Y \subseteq R^p$, the set $Y^{\infty}$ is defined
as follows:
$$Y^{\infty}:=\{d\in R^p : d\neq 0,~y+\alpha d\in Y,~\forall  y\in Y,~\forall\alpha>0\}.$$
In fact, $Y^{\infty}$ is the set of recession directions of
$Y$.
\begin{lemma}\label{lem1}\cite{rock} Assume that $Y\subseteq R^p$ is a closed convex set. If there exist $\hat y\in Y$ and $d\in R^p$ such that $\hat y+\alpha d\in Y$ for any $\alpha>0$, then $d\in Y^\infty$.
\end{lemma}

\begin{lemma}\label{lem2}\cite{rock} Assume that $Y\subseteq R^p$ is a closed convex set. Then $Y$ is unbounded if and only if $Y^\infty\neq
\emptyset$.
\end{lemma}

The following theorem gives a full connection between the notions
addressed in Definition \ref{concepts}.
\begin{theorem}\label{full} Let $Y\subseteq R^p$ be nonempty, $R^p_{\geqq}-$convex and $R^p_{\geqq}-$closed. If
$(y-R^p_\geqq)\cap Y$ is closed
for every $y\in Y$, then the following statements are equivalent: \\
i) $Y_N\neq \emptyset$.\\
ii) $(y-R^p_\geqq)\cap Y$ is bounded for every $y\in
Y.$\\
iii) $Y$ is $R^p_{\geqq}-$compact.\\
iv) $Y$ is $R^p_{\geqq}-$semicompact.\\
v) $Y_N$ is externally stable.\end{theorem}
\begin{proof} (i$\Rightarrow$ii): By contradiction assume that there exists  $\hat y\in Y$ such that
$(\hat y-R^p_\geqq)\cap Y$ is not bounded; thus $(\hat
y-R^p_\geqq)\cap(Y+R^p_\geqq)$ is an unbounded closed convex set.
Hence, by Lemma \ref{lem2} there exists nonzero vector $d\in R^p$
such that $$\hat y+\alpha d\in (\hat
y-R^p_\geqq)\cap(Y+R^p_\geqq),~~~~\forall \alpha>0.$$ Thus $0\neq
d\leqq 0$ which results $y+d\leqq y$ and $y\neq y+d$ for each
$y\in Y+R^p_\geqq$. Furthermore, by Lemma \ref{lem1}, $$y+d\in
Y+R^p_\geqq,~~~\forall y\in Y+R^p_\geqq.$$ These imply
$(Y+R^p_\geqq)_N=\emptyset$. Therefore $Y_N=\emptyset$, due to the
equality of $(Y+R^p_\geqq)_N$ and $Y_N$.
This contradicts the assumption.\\
(ii$\Rightarrow$iii): Trivial.\\
(iii$\Rightarrow$ iv): See Proposition 2.14 in \cite{ehr}.\\
(iv$\Rightarrow$v): This part results from Theorem \ref{semi}.\\
(v$\Rightarrow$ i): Trivial. \end{proof}

\begin{remark}
Notice that some assumptions of the above theorem are redundant in
some parts. Although all assumptions are necessary in proving
$(i)\Rightarrow (ii)$, $\lq\lq R^p_{\geqq}-$convexity and
$R^p_{\geqq}-$closedness" are redundant in proving
$(ii)\Rightarrow (iii)$. To prove $(iii)\Rightarrow (iv)
\Rightarrow (v) \Rightarrow (i)$, three assumptions $\lq\lq
R^p_{\geqq}-$convexity, $R^p_{\geqq}-$closedness, and the
closedness of $(y-R^p_\geqq)\cap Y$" are redundant.
\end{remark}

Now, we deal with a topological property of the nondominated set,
connectedness property. This property may help to explore the
nondominated set starting from a single nondominated point using
local search ideas. Connectedness will also make the task of
selecting a final compromise solution among the set of
nondominated solutions easier \cite{ehr}.

A set $S\subseteq R^p$ is called not connected if there exist open
sets $O_1,~O_2$ such that $S\subseteq O_1\cup O_2$, $S\cap O_1\neq
\emptyset$, $S\cap O_2\neq \emptyset$, and $S\cap O_1\cap
O_2=\emptyset$. Otherwise, $S$ is called connected.

One of the most important results concerning the connectedness of
$Y_N$ has been proved by Naccache \cite{nac}. He established the
following result.

\begin{theorem}\label{connect} \cite{nac} If $Y\subseteq R^p$ is closed, convex, and $R^p_{\geqq}-$compact, then $Y_N$ is
connected.\end{theorem}

This result has been addressed in some text books, see e.g.
Theorem 3.35 in \cite{ehr}. Now, we
show that the $\lq\lq R^p_{\geqq}-$compactness assumption" in Theorem \ref{connect} results from other assumptions of this theorem, and
hence it should be removed. If $Y_N=\emptyset$, then the result is
trivial. So, we assume that $Y_N\neq \emptyset$.

\begin{lemma}  If $Y\subset\mathbb{R}^p$ is closed and convex and $Y_N \neq \emptyset$, then
$Y$ is $R^p_{\geqq}$-compact.
\end{lemma}
\begin{proof}
If $Y$ is not $R^p_{\geqq}$-compact, then there exists some $\hat y\in Y$ such that
$(\hat y-R^p_\geqq)\cap Y$ is unbounded. It is clear that $(\hat y-R^p_\geqq)\cap Y$ is closed and convex.
Hence, by Lemma \ref{lem2} there exists some nonzero vector $d\in R^p$
such that $\hat y+\alpha d\in (\hat
y-R^p_\geqq)\cap Y$ for each $\alpha>0.$ So, $0\neq
d\leqq 0$ and
$$\hat y+\alpha d\in Y,~~\forall\alpha>0.$$
Thus, according to Lemma \ref{lem1}, taking $\hat y\in Y$ into account, we have
$$y+\alpha d\in Y,~~\forall\alpha>0,~\forall y\in Y.$$
This implies $y+d\in Y$ for each $y\in Y.$

Now considering arbitrary $y\in Y$, we have
$$y+d\in Y,~~y+d\leqq y,~~y+d\neq y.$$
This implies $Y_N=\emptyset$.
This contradiction completes the proof.
\end{proof}

The following important corollary results from the above discussion.

\begin{corollary}  If $Y\subset\mathbb{R}^p$ is closed and convex and $Y_N \neq \emptyset$,
then all the following hold:\\
(i) $Y$ is $R^p_{\geqq}$-compact,\\
(ii) $Y$ is $R^p_{\geqq}$-semicompact,\\
(iii) $Y_N$ is externally stable,\\
(iv) $Y_N$ is connected.
\end{corollary}



\begin{thebibliography}{99}



\bibitem{ben} H. Benson, An improved definition of proper efficiency for vector maximization with respect
to cones, J. Optim. Theory Appl. 71 (1979) 232--241.


\bibitem{bor0} J.M. Borwein, Proper efficient points for
maximization with respect to cones, SIAM J. Control Optim. 15
(1977) 57--63.

\bibitem{bor} J.M. Borwein, On the existence of Pareto efficient
points, Mathematics of Operations Research 8 (1983) 64-73.

\bibitem{ehr} M. Ehrgott, Multicriteria Optimization, Springer, Berlin, 2005.


\bibitem{geo} A. Geoffrion, Proper efficiency and the theory of vector maximization, J. Math. Anal.
Appl. 22 (1968) 618--630.


\bibitem{hen} M. Henig, Proper efficiency with respect to cones, J. Optim. Theory Appl.  36(3) (1982) 387--407.


\bibitem{mal} C. Malivert, N. Boissard, Structure of efficient sets
for strictly quasi convex objectives. Journal of Convex Analysis,
1(2) (1994) 143-150.

\bibitem{nac} P. Naccache, Connectedness of the set of nondominated
outcomes in multicriteria optimization. Journal of Optimization
Theory and Applications 25 (1978) 459-467.

\bibitem{rock} R. Rockafellar,
Convex Analysis. Princeton University Press, Pronceton, NJ (1970).

\bibitem{low} T.J. Lowe, J.-F. Thisse, J.E. Ward, R.E.
Wendell, On efficient solutions to multiple objective
mathematical programs, Management Sci. 30 (1984) 1346--1349.

\bibitem{luc}  D.T. Luc, Theory of Vector Optimization, volume 319 of
Lecture Notes in Economics and Mathematical Systems, Springer
Verlag, Berlin 1989.

\bibitem{pop} N. Popovici, Pareto reducible
multicriteria optimization problems, Optimization 54 (2005) 253-263.

\bibitem{saw} Y. Sawaragi, H. Nakayama, and T. Tanino, Theory of Multiobjective
Optimization. Academic Press, Orlando, FL, (1985).

\end{thebibliography}
\end{document}